\documentclass[11pt,onenumbering]{amsart}
\usepackage{comment}
\usepackage[applemac]{inputenc}
\usepackage[square,numbers]{natbib}
\usepackage{amssymb,amsfonts,amscd,mathrsfs,eurosym}
\usepackage{amsmath}
\usepackage{aecompl}
\usepackage{textcomp} 
\usepackage{stmaryrd}
\usepackage{enumerate}
\usepackage{enumitem}
\usepackage{bm}
\usepackage[T1]{fontenc}
\newtheorem{theorem}{Theorem}[section] 
\newtheorem{lemma}[theorem]{Lemma}     
\newtheorem{corollary}[theorem]{Corollary}
\newtheorem{proposition}[theorem]{Proposition}
\newtheorem{remark}[theorem]{Remark}     
\newtheorem{definition}[theorem]{Definition}  
\newcommand{\R}{\mathbb R}

\newcommand{\C}{\mathbb R}

\newcommand{\F}{\mathbb F}

\newcommand{\Ji}{J^{1}}

\DeclareMathOperator{\im}{Im}
\DeclareMathOperator{\rank}{rank}
\begin{document}
\title{Affine geometry of second order ODEs}
\author{Oumar Wone}
\address{Oumar Wone, IMB, 
Institut de Math\'ematiques de Bordeaux, 
Universit\'e Bordeaux I, 351 Cours de la Lib\'eration, 
33405 Talence, France}
\email{oumar.wone@math.u-bordeaux1.fr}
 \maketitle
 \begin{abstract}
 We apply the Cartan equivalence method to the study of real analytic second order ODEs under the local real analytic diffeomorphism of $\C^2$ which are area-preserving. This enables us to give a characterization of the second order ODEs which are equivalent to $y^{\prime\prime}=0$ under such transformations. Moreover we associate to certain of these second order ODEs which satisfy an invariant condition given by the vanishing of a relative differential invariant, an affine normal Cartan connection on the first jet space whose curvature contains all the area-preserving relative differential invariants.
 \end{abstract}
\section{Introduction}
\label{intr}
\noindent

One of the most fundamental problem in the theory of ordinary or partial differential equations is the problem of classification: given two differential equations, does there exist a change of variable transforming one of them to the other?  We will not deal with this problem in full generality; what we will be concerned with is the equivalence problem of second order real analytic scalar equations under the local real analytic transformations of $\C^2$ with identically constant jacobian: $J\equiv 1$. To be more precise, let

\begin{equation}
\label{eqha01}
y^{\prime\prime}=f(x,y,y^{\prime})
\end{equation}
and
\begin{equation}
\label{eqha02}
\overline{y}^{\prime\prime}=\overline{f}(\overline {x},\overline{y},\overline {y}^{\prime})
\end{equation} 
be two scalar second order equations. They are deemed area-preserving equivalent if there exists 
\begin{equation}
\label{eqha03}
\begin{cases}
 \overline{x}=\chi(x,y)\\
  \overline{y}= \phi(x,y)   & \text{with $J:=\chi_{x}\phi_{y}-\chi_{y}\varphi_{x}\equiv 1$}
\end{cases}
\end{equation}
which maps solutions of  \eqref{eqha01} to those of \eqref{eqha02}; this amounts to the fact that the prolongation of \eqref{eqha03}: 
\begin{equation}
\label{eqha04}
\begin{cases}
 \overline{x}=\chi(x,y)\\
  \overline{y}= \phi(x,y)   & \text{with $J:=\chi_{x}\phi_{y}-\chi_{y}\varphi_{x}\equiv 1$}\\
  \overline{y}^{\prime}=\dfrac{\phi_x+y^\prime\phi_y}{\chi_x+y^\prime\chi_y}
\end{cases}
\end{equation}
to the first jet space $\Ji(\C,\C)$ transforms equation \eqref{eqha01} to \eqref{eqha02}.

We will deal with the problem via the method of equivalence of Cartan. This latter method is based of the Cartan theory of exterior differential systems (more specifically pfaffian systems in our case) and it is by encoding the differential geometric equivalence problems into an equivalence problem of coframes or one-forms that one is able to solve the problem through exterior differentiation. Applying this procedure we are able to solve completely the equivalence problem, obtaining two branches of the latter problem, which are distinguished between themselves by the vanishing of the following differential invariant
\begin{equation}
\label{eqha05}
f_y+\dfrac{2}{9}f_{y^{\prime}}^{2}-\dfrac{1}{3}\mathfrak{D}(f_{y^{\prime}}).
\end{equation}
Concretely if the invariant above vanishes identically, we are able to show that the equivalence problem is solved and characterized by the data of an invariant coframe of dimension five. If on the contrary it is not zero, hence does not vanish on some open set of $\Ji$, the solution of the equivalence problem is given by a coframe of length three (this last case has got two subcases). This is done in section \ref{sec4}. In that section beside solving the equivalence problem, we are able to describe by the vanishing of some relative invariants, the necessary and sufficient conditions for a general equation of type \eqref{eqha01} to be brought the equation $y^{\prime\prime}=0$ under a transformation of area preserving type \eqref{eqha03}; we also give the symmetry group of the coframe for the case $y^{\prime\prime}=0$ (see proposition \ref{pr} and corollary \ref{cor}). We remark that there is an instance of part of the result, concerning the linearizability conditions, in the paper \citep{ozawa1999}; but our approach is completely different. The main feature in our present paper is the interpretation we give to the coframe characterizing the equivalence problem in the case of the vanishing of the invariant of equation \eqref{eqha05}. Indeed in this case, we are able to show that there is a canonical (normal) Cartan geometry associated to the equivalence problem, whose associated connection takes values in the Lie algebra $\mathfrak{asl}(2,\C)$ of the Lie group of special affine transformations $ASL(2,\C)$ and the curvature of which contains all the area-preserving relative differential invariants. This gives a geometrization of the equivalence problem with the coframe giving rise to a Cartan connection. See section \ref{sec4} and more precisely theorem \ref{th}. There is one section, section \ref{sec2}, left to us for description; it contains the preliminary definitions and notations that we will use throughout and also the setting up of the equivalence problem we will study.

Finally we remark that there has been important studies of equivalence problems recently \citep{fels1995,kamr1989,merker2006,nur2003,nuro2003} and our paper is a small contribution to this effort. \section{Preliminary notions and setting of the equivalence problem}
\label{sec2}
\subsection{Preliminary notions}
We give a geometric formulation of PDEs throu-gh the fundamental notions of exterior differential systems. In this context, solutions are represented by integral manifolds.
\begin{definition}
\label{defha01}
An exterior differential system (EDS for short) on a manifold $M$ is given by a differential ideal $\mathcal{I}\subset\Omega^\star(M)$: $d\mathcal{I}\subset \mathcal{I}$ and $\mathcal{I}$ is an algebraic ideal with respect to the addition of forms and the wedge product. An integral manifold of the system $\mathcal{I}$ is an immersed submanifold $s:N\to M$ such that $s^\star(\alpha)=0$ for each $\alpha\in\mathcal{I}$.
\end{definition}
Remark that for an exterior differential system we treat all the local coordinates on an equal footing and also do not specify the dimension of the integral submanifolds. The corresponding notion which deals with this matter is given by an exterior differential system with independence condition
\begin{definition}
\label{defha02}
An exterior differential system with independence condition on a manifold $M$ consists of a differential ideal $\mathcal{I}\subset\Omega^\star(M)$ and a non-vanishing decomposable differential $n$-form $\vartheta\subset \Omega^n(M)$ and defined up to scale. This $ \vartheta$ or its equivalence class $\left[\vartheta\right]$ is called the independence condition. The adopted notation for an EDS with independence condition is $\left(\mathcal{I}, \vartheta\right)$.
\end{definition}
Specific to the notion of $EDS$ with independence condition $\left(\mathcal{I}, \vartheta\right)$, there is a corresponding notion of integral manifold given by the following
\begin{definition}
\label{defha03}
An integral manifold \rm{(or solution)} of the system $\left(\mathcal{I},\vartheta\right)$ is an immersed $n$-fold $s:N^n\to M$ such that $s^\star(\alpha)=0$ for all $\alpha$ and $s^\star(\vartheta)≠0$ at each point of $N$.
\end{definition}
For the problem which concerns us and which uses the Cartan equivalence method, we need the more particular notion of linear pfaffian systems as it is the setting in which the equivalence problem is formulated. We remind that a pfaffian system is an EDS generated by one-forms i.e., $\mathcal{I}=\{\theta^a\}_{diff}$ for $\theta^a$ belonging to $\Omega^1(M)$ and $a$ ranging from $1$ to $s$. If $ \vartheta= \vartheta^1\wedge\cdots\wedge \vartheta^n$ represents an independence condition, let $I=\{\theta^a\}$ and $J=\{\theta^a, \vartheta^i\}$ the algebraic ideals generated by the forms between accolades; we will refer to the previous data as $(I,J)$. We have the following
\begin{definition}
\label{defha04}
$(I,J)$ is a linear pfaffian system if $d\theta^a\equiv 0\mod J$ for all $1≤a≤s$.
\end{definition}
Let $(I,J)$ be a linear pfaffian system as above. Let $\pi^\epsilon$, $1≤\epsilon≤\dim M-n-s$ be a collection of one-forms such that $T^\star M$ is locally spanned by $\theta^a,\, \vartheta^i,\,\pi^\epsilon$. Then the linear pfaffian condition is equivalent to the following condition
\begin{equation}
\label{eqha06}
d\theta^a\equiv A^a_{\epsilon i}\pi^\epsilon\wedge \vartheta^i+T_{ij}^a \vartheta^i\wedge \vartheta^j\mod I,
\end{equation}
where we have used the summation convention which consists of summing indices as soon as they occur up and down in a formula. 

We now formulate the equivalence problem of Cartan. Let $U$ and $V$ two open subsets of $\C^m$ and $G\subset GL(m,\C)$ a Lie subgroup, that is an immersed submanifold which at the same time a subgroup of $GL(m,\C)$. Let $\varepsilon$ and $\overline{\varepsilon}$ respective coframes on $U$ and $V$. The equivalence problem consists of determining all the diffeomorphisms $h:U\to V$ such that
\begin{equation}
\label{eqha07}
h^\star(\overline{\varepsilon}^i)=\lambda_j^i\epsilon^j
\end{equation} 
where $(\gamma)_i^j:U\to G$ and the indices $i,\,j$ range from $1$ to $m$. The fundamental idea of Cartan in dealing with the problem is to work on the cartesian products $U\times G$ and $V\times G$ endowed with the respective one-forms
\begin{equation}
\label{eqha08}
\Upsilon_{(p,S)}^i=S_j^i\pi^\star_U(\varepsilon_p^j)\quad \text{and}\quad \overline{\Upsilon}^i_{(q,T)}=T_j^i\pi_V^\star(\overline{\varepsilon}^j_q),
\end{equation}
with $\pi_U:U\times G\to U$ and $\pi_V:V\times G\to V$, the canonical projections. The equivalence problem amounts now to the determination of diffeomorphisms $\widetilde{h}:U\times G\to V\times G$ such that
$$\widetilde{h}^\star\overline{\Upsilon}^i=\Upsilon.$$
One can then interpret the equivalence problem as the local equivalence problem of $G$-structures, which we first of all define.

Let $\Sigma$ be a manifold of dimension $N$ and let $\F(\Sigma)$ the bundle of linear frames of $\Sigma$. $\F(\Sigma)$ is a principal $GL(N,\C)$-bundle over $\Sigma$ with the left $GL(N,\C)$ action given by $u\to S.u$. A point $u\in\F(\Sigma)$ is a basis of the tangent space $T_p\Sigma$ where $p=\pi(u)$ and $\pi:\F(\Sigma)\to\Sigma$ denotes the canonical projection onto the base. Fix a real vector space $W$ of dimension $N$; an element of $\F(\Sigma)$ thus defines an isomorphism, which we shall denote also by $u$, of $T_p\Sigma$ onto $V$.
\begin{definition}
\label{defha05}
A $G$-structure $B_G$ over $\Sigma$ is a reduction of $\F(\Sigma)$ to the group $G$, that is a submanifold of $\F(\Sigma)$ such for all $u\in B_G$ and all $S\in GL(W)$, the point $S.u$ lies in $B_G$ if and only if $S\in G$. We again denote by $\pi$ the restriction of the canonical projection to $B_G$.
\end{definition}
On $\F(\Sigma)$, there sits a tautological one form, a smooth section of the bundle $T^\star\F(\Sigma)\otimes V$ of $V$-valued one-forms on $\F(\Sigma)$. It is defined as follows. Take $u\in\F(\Sigma)$ and $X\in T_u\F(\Sigma)$; then $\pi_\star(X)\in T_p\Sigma$, where $p=\pi(u)$ and we set
$$\Upsilon=u(\pi_\star X).$$
As $B_G$ is a submanifold of $\F(\Sigma)$, the restriction of the form $\Upsilon$ to $B_G$ is again $V$-valued. We introduce the notions of equivalence and local equivalence of $G$-structures.
\begin{definition}
Two $G$-structures $B_G^1$ and $B_G^2$ over manifolds $\Sigma_1$ and $\Sigma_2$ are called equivalent or isomorphic if there exists a diffeomorphism $h:\Sigma_1\to\Sigma_2$ such that $h_\star B^1_G=B_G^2$. They are called locally isomorphic at $(p,q)\in M_1\times M_2$ if there exists open neighborhoods $U$ and $V$ of $p$ and $q$ such that $B_G^1|_U$ and $B_G^2|_V$ are isomorphic.
\end{definition}
The fundamental remark is that the local equivalence for $G$-structures and the Cartan equivalence problem are all but one thing (when we lift the equivalence problem). Concretely two $G$-structures $B_G^1$ with canonical one-form $\Upsilon$ and $B_G^2$ with canonical one-form $\overline{\Upsilon}$ are locally equivalent if and only if a local diffeomorphism $\widetilde{h}:B_G^1|_U\to B_G^2|_V$ pullbacks the tautological one-form $\overline{\Upsilon}|_V$ of $B_G^2|_V$ to the canonical form $\Upsilon|_U$ of $B_G^1|_U$. This is exactly the lifted Cartan equivalence problem if we choose sufficiently small open sets $U$ and $V$ where $B_G^1$ and $B_G^2$ are trivial. Some last remarks about the different steps of the equivalence method, before we tackle our problem at hand, are the following.

The exterior derivative of each of the one-forms of equation \eqref{eqha08} gives what is called the structures equations. They are given in components by the formulas 
$$d\Upsilon^i=a_{j\rho}^i\varpi^\rho\wedge\Upsilon^j+\dfrac{1}{2}\gamma^i_{jk}\Upsilon^j\wedge\Upsilon^k,1≤i,j,k≤m,\,1≤\rho≤r:=\dim G$$
and
$$d\overline{\Upsilon}^i=a_{j\rho}^i\varpi^\rho\wedge\overline{\Upsilon}^j+\dfrac{1}{2}\overline{\gamma}^i_{jk}\overline{\Upsilon}^j\wedge\overline{\Upsilon}^k,$$
with $\gamma_{jk}^i$ and $\overline{\gamma}_{jk}^i$ anti-symmetric in their lower indices. Let $F$ be an $m$-dimensional real vector of dimension $m$ with a basis $\{e_i,\,1≤i≤m\}$ and $F^\star$ its dual, endowed with the dual basis $\{e_i^\star,\,1≤i≤m\}$. Also, consider $\{\delta_e,\,1≤\rho≤r\}$ the basis of $T_eG\simeq g$ dual the basis $\{\varpi^\rho_e,\,1≤\rho≤r\}$ of $T^\star_eG\simeq g^\star$. Define the map $L:g\otimes V^\star\to V\otimes\Lambda^2V^\star:v_i^\rho\delta_e\otimes e^{\star i}\to (a_{j\rho}^iv_{k}^\rho-a_{k\rho}^iv_j^\rho)e_i\otimes e^{\star j}\wedge e^{\star k}$ and the exact kernel-cokernel sequence
\begin{equation}
0\to\mathfrak{g}^{(1)}\overset{i}{\to}\mathfrak{g}\otimes V^{\star}\overset{L}{\to}V\otimes\Lambda^2V^{*}\overset{pr}{\to}\Pi_\mathfrak{g}\to 0
\end{equation}
where
\begin{equation}
\label{eqci09}
\mathfrak{g}^{(1)}=\ker L, \,\Pi_\mathfrak{g}=V\otimes\Lambda^2V^{*}/\im L:=H^{0,2}(\mathfrak{g}).
\end{equation}
$\mathfrak{g}^{(1)}$ is known the first prolongation of the Lie-algebra $\mathfrak{g}$ and $H^{0,2}(\mathfrak{g})$ the space of intrinsic torsion. The first-order structure tensor or intrinsic torsion is defined by
$\tau:U\times G\to H^{0,2}(\mathfrak{g}):(p,S)\to (pr\circ\gamma)(p,S)$ with $pr$ the canonical projection onto $H^{0,2}(\mathfrak{g})$ and $\gamma$ is given by
$$\gamma:U\times G\to V\otimes\Lambda^2V^\star:(p,S)\to\dfrac{1}{2}\gamma_{jk}^ie_i\otimes e^{\star j}\wedge e^{\star k}.$$
As Gardner \citep{gard1989} has showed, one effective way to determine the structure tensor is to do the so-called Lie-algebra compatible absorption of torsion. The two remaining steps are in the problem of equivalence are the processes of reduction and prolongation, which consist heuristically, for the reduction case, of restricting to a subbundle of the associated $G$-structure $B_G$, by well-choosing certain parameters of the group $G$. The process of prolongation on the other hand deals with the possibility of the presence of higher order conditions of equivalence. It is by enlarging the space in a geometrical manner that it takes those possibilities into account. Finally we know from the Cartan-Kuranishi theorem \citep{kuran1957} that the equivalence problem gives always a solution in the form of an involutive system (Cartan-Kähler \citep{gard1991}) or via an $\{e\}$-structure, when one is successful in reducing all the group parameters.

Our references hereafter are \citep{stern1964,kamr1989,olv1995} and we refer the reader to them from a complete exposition of the details of the equivalence method; it is their formalism that we shall adopt here.
 
\subsection{Setting of the equivalence problem}

In all this paper, $\mathfrak{D}$ and $\overline{\mathfrak{D}}$, will denote the total derivatives
$$\mathfrak{D}:=\partial_x+y^\prime\partial_y+f\partial_{y^\prime}$$
and 
$$\overline{\mathfrak{D}}=\partial_{\overline{x}}+\overline{y}^\prime\partial_{\overline{y}}+\overline{f}\partial_{\overline{y}^\prime}$$

Consider two second ODEs
\begin{equation}
\label{eqha10}
y^{\prime\prime}=f(x,y,y^{\prime})
\end{equation}
and
\begin{equation}
\label{eqha11}
\overline{y}^{\prime\prime}=\overline{f}(\overline{x},\overline{y},\overline{y}^\prime),
\end{equation}
on two copies of $\Ji(\C,\C)$ and the pseudo-group of local $C^\omega$ diffeormorphisms
\begin{equation*}
\label{eqha12}
\begin{cases}
 \overline{x}=\chi(x,y)\\
  \overline{y}= \phi(x,y)   & \text{with $J:=\chi_{x}\phi_{y}-\chi_{y}\varphi_{x}\equiv 1$}.
\end{cases}
\end{equation*}
We remind that our goal is the study of these two differential equations under the above prescribed pseudo-group of equation \eqref{eqha03}. One first remarks that the setting (see equation \eqref{eqha08}) of the equivalence is completely symmetric. So for its study we can restrict only to \eqref{eqha10}. In order to apply the Cartan equivalence method we discussed in the previous section, one has to express the differential equation \rm{(\ref{eqha10})} into a pfaffian system; this is achieved by the following result.
\begin{lemma}
\label{lem1}
The solutions of equation \eqref{eqha10} are in one-to-one correspondence with the integral manifolds of the pfaffian system with independence condition on $\omega^3$: $(I_f, \omega^{3}:= dx)$ on $\Ji(\C,\C)$ where $I_f$ is generated as a differential ideal, relatively to the wedge product, by the $1$-forms $$\omega^{1}:=dy-y^{\prime}dx,\,\omega^{2}:=dy^{\prime}-fdx$$ and $(x,y,y^{\prime})$ are standard coordinates on $\Ji(\C,\C)$ in which the contact form is given by $\omega^{1}$.
\end{lemma}
 
\begin{proof}
An integral manifold of $(I_{f},\omega^{3})$ with independence condition on $\omega^3$
is by very definition a curve $s:\C\to \Ji(\C,\C)$ such that $s^{\star}\omega^{1}=s^{\star}\omega^{2}=0$ and $s^{\star}\omega^{3}\neq 0$. The conditions $s^{\star}\omega^{1}=0$ and $s^{\star}\omega^{3}\neq 0$ ($x$ may be chosen as a local coordinate on $\R$) will be satisfied if and only there exists a map $c:\C\to \C: x\to c(x)$ such that $s=j^{1}c$ where 
$$j^1c:\C\to\Ji(\C,\C):x\to\left(x,c(x),\dfrac{dc}{dx}\right).$$
The condition $s^{\star}\omega^{2}=(j^1c)^{\star}\omega^{2}=0$ is then equivalent to 
$$\dfrac{d^2c}{dx^2}=f\left(x,c,\dfrac{dc}{dx}\right).$$
\end{proof}

Consider now a local local diffeomorphism of the prescribed type $\Psi:\Ji(\C,\C)\to \Ji(\C,\C)$ and let $\psi:\C\to\C$ the transformation induced by $\Psi$ through the fibration $\alpha: \Ji(\C,\C)\to\C$ over the source. In other words $\psi$ is defined by the requirement that the diagram 
\begin{equation}
\label{eqha14}
\begin{CD}
\Ji(\C,\C) @>\Psi>> \Ji(\C,\C)\\
@VV\alpha V @VV\alpha V\\
\C @>\psi>> \C
\end{CD}
\end{equation}
commutes.

Using the lemma \ref{lem1}, we see that $\Psi$ will map the $1$-jet of a solution of equation \eqref{eqha10} to the $1$-jet of a solution of equation \eqref{eqha11} if and only if $\Psi\circ s\circ \psi^{-1}$ is an integral manifold of $(I_{F},\overline{\omega}^3)$ as soon as $s$ is an integral manifold of $(I_f,\omega^3)$. This is equivalent to $\Psi$ satisfying
\begin{equation}
\label{eqha15}
\Psi^\star\left(\begin{array}{c}\overline{\omega}^1 \\\overline{\omega}^2 \\\overline{\omega}^3\end{array}\right)=\left(\begin{array}{ccc}a & b & 0 \\c & d & 0 \\e & f & g\end{array}\right)\left(\begin{array}{c}\omega^1 \\\omega^2 \\\omega^3\end{array}\right), \quad \text{with $f(ad-bc)≠0$.}
\end{equation}
We must now express the fact that $\Psi$ is induced by a point transformation, in other words that $\Psi$ is the first prolongation $p^1\Phi$ of an element of $C^\omega_{loc}(J^0(\C,\C))$. We recall that $p^1\Phi$ is defined by the condition that the diagram
\begin{equation}
\label{eqha16}
\begin{CD}
\Ji(\C,\C) @>p^1\Phi>> \Ji(\C,\C)\\
@VV\pi_0^1 V @VV\pi_0^1 V\\
J^0(\C,\C) @>\Phi>> J^0(\C,\C)
\end{CD}
\end{equation}
be commutative and that the contact systems be preserved, that is
\begin{equation}
\label{eqha17}
(p^1\Phi)^\star\overline{\omega}^1=\kappa\omega^1
\end{equation}
with $\kappa:\Ji(\C,\C)\to \C$ a nowhere vanishing function. This is achieved by the following lemma
\begin{lemma}
A necessary and sufficient condition for a local real analytic of area-preserving type $\Psi:\Ji(\C,\C)\to \Ji(\C,\C)$ to be the first prolongation of a local real analytic $\Phi:J^0(\C,\C)\to J^0(\C,\C)$ is that 
 
\begin{equation}
\label{eqha18}
\Psi^\star\left(\begin{array}{c}\overline{\omega}^1 \\\overline{\omega}^2 \\\overline{\omega}^3\end{array}\right)=\left(\begin{array}{ccc}u_1 & 0 & 0 \\u_2 & u_1^2 & u_4 \\u_3 & 0 & \dfrac{1}{u_1}\end{array}\right)\left(\begin{array}{c}\omega^1 \\\omega^2 \\\omega^3\end{array}\right), \quad \text{with $u_1≠0$.}
\end{equation}
\end{lemma}
\begin{proof}
The commutative diagram \eqref{eqha16} and the preservation of contact forms means that $p^1\Phi$ is of the form
$$p^1\Phi:(x,y,y^\prime)\to\left(\chi(x,y),\phi(x,y),\dfrac{\phi_x+y^\prime\phi_y}{\chi_x+y^\prime\chi_y}(x,y)\right).$$
Pulling back the overlined one-forms, one gets the equation \eqref{eqha18} with $u_1=\dfrac{1}{\chi_x+y^\prime\chi_y}$. $\chi_x+y^\prime\chi_y$ does not vanish as $p^1\Phi$ is a local diffeormorphism. Conversely $(p^1\Phi)^\star\overline{\omega}^3\equiv 0$ modulo the ideal generated by $\omega^2$ and $\omega^3$ implies that the diagram \eqref{eqha16} commutes; $(p^1\Phi)^\star\overline{\omega}^1\equiv 0$ modulo the ideal generated by $\omega^1$ implies that the contact systems are preserved and finally equation \eqref{eqha18} gives the area conservation condition as it leads to $(p^1\Phi)^\star\overline{\omega}^1\wedge\overline{\omega}^2=\omega^1\wedge\omega^2$.  
\end{proof}
Combining equations \eqref{eqha15} and \eqref{eqha18}, we get the following lemma which encodes all the information about the local equivalence of two second order ODEs.
\begin{lemma}
A local transformation of $C^\omega$ type $\Psi:\Ji(\C,\C)\to\Ji(\C,\C)$ satisfies the equations \eqref{eqha16} and \eqref{eqha18} if and only if it solves the Cartan equivalence problem
\begin{equation}
\label{eqha19}
\Psi^\star\left(\begin{array}{c}\overline{\omega}^1 \\\overline{\omega}^2 \\\overline{\omega}^3\end{array}\right)=\left(\begin{array}{ccc}u_1 & 0 & 0 \\u_2 & u_1^2 & 0 \\u_3 & 0 & \dfrac{1}{u_1}\end{array}\right)\left(\begin{array}{c}\omega^1 \\\omega^2 \\\omega^3\end{array}\right), \quad \text{with $u_1≠0$.}
\end{equation}
\end{lemma}

\section{Solution of the Equivalence Problem}
\label{sec3}
We consider the Cartan equivalence problem with $3$-dimensional structural group $G$ given by
\begin{equation}
\label{eqha20}
G=\left(\begin{array}{ccc}u_1 & 0 & 0 \\u_2 & u^2_1 & 0 \\u_3 & 0 & \dfrac{1}{u_1}\end{array}\right),\quad u_1≠0.
\end{equation}
As explained in the previous section in order to solve the equivalence problem, one has to work on the $G$-structure $B_G$ over $\Ji(\R,\R)$, and it is by differentiating the tautological one-form which sits on it that one is able to derive conditions of equivalence. This one-form is given in local coordinates by
\begin{equation}
\label{eqha21}
\theta=\left(\begin{array}{c}\theta^1 \\\theta^2 \\\theta^3\end{array}\right)=\left(\begin{array}{ccc}u_1 & 0 & 0 \\u_2 & u_1^2 & 0 \\u_3 & 0 & \dfrac{1}{u_1}\end{array}\right)\left(\begin{array}{c}\omega^1 \\\omega^2 \\\omega^3\end{array}\right)
\end{equation}
 The first step towards the solution of any Cartan equivalence problem is to determine the so-called structure equations. In our case they are given by the below formulas
 \begin{equation}
 \label{eqha22}
 \begin{split}
 d\theta^1&=\pi^1\wedge\theta^1+T^1_{31}\theta^3\wedge\theta^1+T^1_{32}\theta^3\wedge\theta^2+T^1_{21}\theta^2\wedge\theta^1\\
 d\theta^2&=\pi^2\wedge\theta^1+2\pi^1\wedge\theta^2+T^2_{31}\theta^3\wedge\theta^1+T^2_{32}\theta^3\wedge\theta^2+T^2_{21}\theta^2\wedge\theta^1\\
 d\theta^3&=\pi^3\wedge\theta^1-\pi^1\wedge\theta^2+T^3_{31}\theta^3\wedge\theta^1+T^3_{32}\theta^3\wedge\theta^2+T^3_{21}\theta^2\wedge\theta^1
 \end{split}
\end{equation}
where $\pi^1$, $\pi^2$ and $\pi^3$ are components of the right-invariant Maurer-Cartan matrix of the Lie subgroup $G$ of equation \eqref{eqha20}. The latter matrix is given by
\begin{equation}
\label{eqha23}
\left(\begin{array}{ccc}\pi^1 & 0 & 0 \\\pi^2 & 2\pi^1 & 0 \\\pi^3 & 0 & -\pi^1\end{array}\right)=\left(\begin{array}{ccc}\frac{du_1}{u_1} & 0 & 0 \\\frac{du_2}{u_1}-2u_2\frac{du_1}{u_1^2} & 2\frac{du_1}{u_1} & 0 \\\frac{du_3}{u_1}+u_3\frac{du_1}{u_1^2} & 0 & -\frac{du_1}{u_1^2}\end{array}\right).
\end{equation} 
We absorb the torsion without changing how we denote the components of the Maurer-Cartan matrix (modulo basic forms). That we set
\begin{equation}
\label{eqha13}
\begin{split}
\pi^1&\to\pi^1+T_{31}^1\theta^3+T_{21}^1\theta^2\\
\pi^2&\to\pi^2+T_{31}^2\theta^3+T_{21}^2\theta^2\\
\pi^3&\to\pi^3+T_{31}^3\theta^3+T_{21}^3\theta^2
\end{split}
\end{equation}
The new structure equations are given by
 \begin{equation}
 \label{eqha24}
 \begin{split}
d\theta^1&=\pi^1\wedge\theta^1+T^1_{32}\theta^3\wedge\theta^2\\
 d\theta^2&=\pi^2\wedge\theta^1+2\pi^1\wedge\theta^2+\overline{T}^2_{31}\theta^3\wedge\theta^1\\
 d\theta^3&=\pi^3\wedge\theta^1-\pi^1\wedge\theta^2.
 \end{split}
\end{equation}
$T^1_{32}=1$ and $\overline{T}^2_{31}=u_1f_{y^\prime}+3\dfrac{u_2}{u_1}$. We restrict to the subbundle $B_{G_{(1)}}$ of $B_G$ characterized by the fact that $\overline{T}^2_{31}\equiv 0$. This gives
\begin{equation}
\label{eqha25}
G_{(1)}=\left(\begin{array}{ccc}u_1 & 0 & 0 \\0 & u_1^2 & 0 \\u_3 & 0 & \dfrac{1}{u_1}\end{array}\right)
\end{equation}
and the tautological one-form is given in local coordinates by
\begin{equation}
\label{eqha26}
\theta=\left(\begin{array}{c}\theta^1 \\\theta^2 \\\theta^3\end{array}\right)=\left(\begin{array}{ccc}u_1 & 0 & 0 \\0 & u_1^2 & 0 \\u_3 & 0 & \dfrac{1}{u_1}\end{array}\right)\left(\begin{array}{c}\omega^1 \\\omega^2 \\\omega^3\end{array}\right)
\end{equation}
with $\omega^2$ assuming now the local expression
$$\omega^2=dy^\prime-fdx-\dfrac{1}{3}f_{y^\prime}(dy-y^\prime dx).$$
We work parametrically. The structure equations on $B_{G_{(1)}}$ are locally given by
\begin{equation}
\label{eqha27}
\begin{split}
d\theta^1&=\dfrac{du_1}{u_1}\wedge\theta^1+\theta^3\wedge\theta^2+\dfrac{u_3}{u_1}\theta^2\wedge\theta^1+\dfrac{1}{3}u_1f_{y^{\prime}}\theta^3\wedge\theta^1\\
d\theta^2&=2\dfrac{du_1}{u_1}\wedge\theta^2+\dfrac{2}{3}u_1f_{y^{\prime}}\theta^3\wedge\theta^2+\dfrac{1}{3u_1}\left(f_{y^{\prime}y^{\prime}}-2u_1u_3f_{y^{\prime}}\right)\theta^1\wedge\theta^2\\
                 &+u_1^2\left(f_y+\dfrac{2}{9}f_{y^{\prime}}^{2}-\dfrac{1}{3}\mathfrak{D}(f_{y^{\prime}})\right)\theta^3\wedge\theta^1\\
d\theta^3&=\dfrac{u_3}{u_1}\left(\dfrac{du_3}{u_3}+\dfrac{du_1}{u_1}\right)\wedge\theta^1-\dfrac{du_1}{u_1}\wedge\theta^3+\dfrac{u_3}{u_1}\theta^3\wedge\theta^2+\left(\dfrac{u_3}{u_1}\right)^2\theta^2\wedge\theta^1\\
                 &+\dfrac{1}{3}u_3f_{y^{\prime}}\theta^{3}\wedge\theta^{1};                 
\end{split}
\end{equation} 

after collection of the torsion, we set 
\begin{equation}
\label{eqha28}
\begin{split}
\pi^1&=\dfrac{du_1}{u_1}+\dfrac{u_3}{u_1}\theta^{2}+\dfrac{1}{3}u_1f_{y^{\prime}}\theta^3+\dfrac{1}{6u_1}\left(f_{y^\prime y^\prime}-2u_1u_3f_{y^{\prime}}\right)\theta^1\\
\pi^2&=\dfrac{u_3}{u_1}\left(\dfrac{du_1}{u_1}+\dfrac{du_3}{u_3}\right)+\left(\dfrac{u_3}{u_1}\right)^2\theta^2+\dfrac{1}{3}u_3f_{y^\prime}\theta^3-\dfrac{1}{6u_1}\left(f_{y^{\prime}y^{\prime}}-2u_1u_3f_{y^{\prime}}\right)\theta^3.
\end{split}
\end{equation}
This gives us
\begin{equation}
\label{eqha29}
\begin{split}
d\theta^1&=\pi^1\wedge\theta^1+\theta^3\wedge\theta^2\\
d\theta^2&=2\pi^1\wedge\theta^2+u_1^2\left(f_y+\dfrac{2}{9}f_{y^{\prime}}^{2}-\dfrac{1}{3}\mathfrak{D}(f_{y^{\prime}})\right)\theta^{3}\wedge\theta^{1}\\
d\theta_3&=\pi^2\wedge\theta^1-\pi^1\wedge\theta^3.
\end{split}
\end{equation}

Assume $f_y+\dfrac{2}{9}f_{y^{\prime}}^{2}-\dfrac{1}{3}\mathfrak{D}(f_{y^{\prime}})$ does not vanish anywhere (we can arrange to be in this situation as we are working locally). Then it takes locally, a strictly positive or strictly negative value. If its value is strictly negative then we reduce to the subbundle of $B_{G_{(1)}}$ characterized by the fact that
$$u_1^2\left(f_y+\dfrac{2}{9}f_{y^{\prime}}^{2}-\dfrac{1}{3}\mathfrak{D}(f_{y^{\prime}})\right)\equiv 1.$$
If on the contrary its value is strictly negative then we restrict to the subbundle of $B_{G_{(1)}}$ where the relation
$$u_1^2\left(f_y+\dfrac{2}{9}f_{y^{\prime}}^{2}-\dfrac{1}{3}\mathfrak{D}(f_{y^{\prime}})\right)\equiv -1$$
is satisfied. In any of the two cases, we see easily using equation \eqref{eqha27} that we can further reduce to the subbundle characterized by
$$u_3=0,$$
thus we get an $\{e\}$-structure $B_{\{e\}}\simeq \Ji(\R,\R)$. We summarize the preceding discussion in the following proposition
\begin{proposition}
 If the relative invariant $f_y+\dfrac{2}{9}f_{y^{\prime}}^{2}-\dfrac{1}{3}\mathfrak{D}(f_{y^{\prime}})\equiv 0$, then the Cartan equivalence method yields an $\{e\}$-structure or complete parallelism on $\Ji(\R,\R)$.
 \end{proposition}
 
 Suppose now that the equation
\begin{equation}
\label{eqha30}
f_y+\dfrac{2}{9}f_{y^{\prime}}^{2}-\dfrac{1}{3}\mathfrak{D}(f_{y^{\prime}})\equiv 0
\end{equation}
holds. One sees right away from the equations \eqref{eqha29} that the structure tensor has only constant components, yet the structure group of $B_{G_{(1)}}$ (equation \eqref{eqha25}) is not the identity. In order to continue the procedure one has to make the arithmetic test devised by Cartan, the so-called Cartan involutivity test. If the Lie algebra $\mathfrak{g}_{(1)}$ is involutive then we deduce that any $G_{(1)}$-structures are equivalent. In on the contrary it is not then on has to prolong the equivalence problem because there is a possibility there remains higher order differential invariants. This is how the Cartan test for our situation is done. Let $v\in\R^3$ and $L\left[v\right]$ the $3\times 2$ matrix  
\begin{equation}
\label{eqha31}
L\left[v\right]= \left(\begin{array}{cc}v^1 & 0 \\2v^2 & 0 \\-v^3 & v^1\end{array}\right).
\end{equation} 
The first reduced Cartan character $s^\prime_1$ is defined by
$$s^\prime_1=\max\{\rank L\left[v\right], v\in\R^3\},$$
the second reduced Cartan character is given by
$$s^\prime_1+s^\prime_2=\max\left\{\rank \left\{\begin{array}{c}L\left[v_1\right]\\ L\left[v_2\right]\end{array}\right\}, v_1,v_2\in\R^3\right\},$$
and 
$$s^\prime_1+s^\prime_2+s^\prime_3=2.$$
Setting $v=(1,0,-1)$ we see $s^\prime_1=2$; therefore $s^\prime_2=0=s^\prime_3=0$. We must compute the dimension of the Lie algebra $\mathfrak{g}_{(1)}^{(1)}$. This is done by determining the ambiguity in the choice of the forms $\pi^1$ and $\pi^2$ in equation \eqref{eqha29}. One readily sees that it is the transformation

\begin{equation}
\label{eqha32}
\left(\begin{array}{c}\pi^1 \\\pi^2\end{array}\right)\to \left(\begin{array}{c}\pi^1 \\\pi^2\end{array}\right)+\left(\begin{array}{ccc}0 & 0 & 0 \\t_1 & 0 & 0\end{array}\right) \left(\begin{array}{c}\theta^1 \\\theta^2 \\\theta^3\end{array}\right)
\end{equation}
where $t_1$ is an arbitrary function on $B_{G_{(1)}}$, that leaves invariant the structure equations \eqref{eqha29}. This gives $\dim \mathfrak{g}_{(1)}^{(1)}=1$. We make the Cartan test:
$$\dim \mathfrak{g}_{(1)}^{(1)}< s^\prime_1+2s^\prime_2+3s_3^\prime=2.$$

Thus we have to prolongate the equivalence problem. 

The prolongated equivalence problem is defined on the $G_{(1)}^{(1)}$-structure $B_{G_{(1)}^{(1)}}$ with associated structure group 
\begin{equation}
\label{eqha33}
G_{(1)}^{(1)}=\left(\begin{array}{ccccc}1 & 0 & 0 & 0 & 0 \\0 & 1 & 0 & 0 & 0 \\0 & 0 & 1 & 0 & 0 \\0 & 0 & 0 & 1 & 0 \\t_1 & 0 & 0 & 0 & 1\end{array}\right)
\end{equation}
and canonical one-form
\begin{equation}
\label{eqha34}
\left(\begin{array}{c}\overline{\theta}^{1} \\\overline{\theta}^{2} \\\overline{\theta}^{3} \\\Omega^1 \\\Omega^2\end{array}\right)=\left(\begin{array}{ccccc}1 & 0 & 0 & 0 & 0 \\0 & 1 & 0 & 0 & 0 \\0 & 0 & 1 & 0 & 0 \\0 & 0 & 0 & 1 & 0 \\t_1 & 0 & 0 & 0 & 1\end{array}\right)\left(\begin{array}{c}\theta^1 \\\theta^2 \\\theta^3 \\\pi^1 \\\pi^2\end{array}\right)
\end{equation}

We compute $d\Omega_1$:
\begin{equation}
\label{eqha35}
\begin{split}
d\Omega^1&=\Omega^2\wedge\theta^2\\
&+\left(t_1+\dfrac{1}{6u_1^3}f_{y^\prime y^\prime y\prime}+\dfrac{2}{3}\dfrac{u^2_3}{u_1}f_{y^\prime}-\dfrac{1}{2}\dfrac{u_3}{u_1^2}f_{y^\prime y^\prime}\right)\theta^2\wedge\theta^1\\
&+\left(-\dfrac{1}{3}f_{y^\prime y}-\dfrac{1}{18}f_{y^\prime}f_{y^\prime y^\prime}+\dfrac{1}{6}\mathfrak{D}(f_{y^\prime y^\prime})\right)\theta^3\wedge\theta^1\\
&+u_1u_3\left(f_y-\dfrac{1}{3}\mathfrak{D}(f_{y^{\prime}})+\dfrac{2}{9}f_{y^{\prime}}^2\right)\theta^3\wedge\theta^1.
\end{split}
\end{equation}
The computation of $d\Omega^2$ gives
\begin{equation}
\label{eqha36}
\begin{split}
d\Omega^{2}&=dt_1\wedge\theta^1+3t_1\Omega^1\wedge\theta^1+2\Omega^2\wedge\Omega^1\\
&+\left(-\dfrac{1}{2u_1}f_{y^{\prime}y^{\prime}}+\dfrac{4}{3}u_3f_{y^{\prime}}\right)\Omega^2\wedge\theta^1\\
&+\left(t_1+\dfrac{2}{3}\dfrac{u_3^2}{u_1}f_{y^\prime}+\dfrac{1}{6u_1^3}f_{y^{\prime}y^{\prime}y^{\prime}}-\dfrac{1}{2}f_{y^{\prime}y^{\prime}}\dfrac{u_3}{u_1^2}\right)\theta^3\wedge\theta^2\\
&-\dfrac{u_3^2}{3u_1^3}\left(f_{y^\prime y^{\prime}}-2u_1u_3f_{y^\prime}\right)\theta^2\wedge\theta^1\\
&+\dfrac{1}{u_1^2}\left(\dfrac{1}{6}\left(f_{y^{\prime}y^{\prime}y}+\dfrac{1}{3}f_{y^{\prime}}f_{y^{\prime}y^{\prime}y^{\prime}}\right)-\dfrac{1}{36}f_{y^{\prime}y^{\prime}}^2\right)\theta^3\wedge\theta^1\\
&+\dfrac{u_3}{u_1}\left(-\dfrac{1}{18}f_{y^{\prime}}f_{y^{\prime}y^{\prime}}-\dfrac{1}{6}\mathfrak{D}(f_{y^{\prime}y^{\prime}})-\dfrac{2}{3}f_{y^{\prime}y}\right)\theta^3\wedge\theta^1\\
&+u_3^2\left(f_y+\frac{1}{3}\mathfrak{D}(f_{y^{\prime}})\right)\theta^3\wedge\theta^1.
\end{split}
\end{equation}

We recollect the torsion and do reduction by setting $$t_1=-\dfrac{2}{3}\dfrac{u_3^2}{u_1}f_{y^\prime}-\dfrac{1}{6u_1^3}f_{y^{\prime}y^{\prime}y^{\prime}}+\dfrac{1}{2}\dfrac{u_3}{u_1^2}f_{y^{\prime}y^{\prime}};$$ we obtain an $\{e\}$-structure on a bundle $B_{G_{(2)}^{(1)}}$ of dimension $5$ with structure equations

\begin{equation}
\label{eqha37}
\begin{split}
d\theta^{1}&=\Omega^{1}\wedge\theta^{1}+\theta^3\wedge\theta^2\\
d\theta^{2}&=2\Omega^1\wedge\theta^2\\
d\theta^{3}&=\Omega^{2}\wedge\theta^{1}-\Omega^{1}\wedge\theta^{3}\\
d\Omega^1&=\Omega^2\wedge\theta^2+\left(-\dfrac{1}{3}f_{y^\prime y}-\dfrac{1}{18}f_{y^\prime}f_{y^\prime y^\prime}+\dfrac{1}{6}\mathfrak{D}(f_{y^\prime y^\prime})\right)\theta^3\wedge\theta^1\\
 d\Omega^{2}&=2\Omega^2\wedge\Omega^1-\dfrac{1}{6u_1^5}f_{y^\prime y^\prime y^\prime y^\prime}\theta^2\wedge\theta^1 \\
                         &+\dfrac{1}{u_{1}^2}\left(\dfrac{1}{6}f_{y^{\prime}y^{\prime}y}-\dfrac{1}{6}\mathfrak{D}(f_{y^{\prime}y^{\prime}y^{\prime}})-\dfrac{1}{9}f_{y^{\prime}}f_{y^{\prime}y^{\prime}y^{\prime}}+\dfrac{1}{18}f_{y^{\prime}y^{\prime}}^2\right)\theta^3\wedge\theta^1\\
                         &+2\dfrac{u_3}{u_1}\left(-\dfrac{1}{18}f_{y^{\prime}}f_{y^{\prime}y^{\prime}}+\dfrac{1}{6}\mathfrak{D}(f_{y^{\prime}y^{\prime}})-\dfrac{1}{3}f_{y^{\prime}y}\right)\theta^3\wedge\theta^1
 \end{split}
\end{equation}
Let $I_1$ be the coefficient of $\theta^3\wedge\theta^1$ in $d\Omega^1$, $I_2$ respectively $I_3$ the coefficient of $\theta^2\wedge\theta^1$ respectively the one of $\theta^3\wedge\theta^1$ in $d\Omega^2$. Taking the exterior derivatives $d^2\Omega^1$ and $d^2\Omega^2$, one has the following relations between $I_1$, $I_2$, $I_3$
\begin{equation}
\label{eqha38}
\begin{split}
dI_1+I_3\theta^2&\equiv 0\mod \theta^1, \theta^3\\
dI_3+2I_3\Omega^1-2I_1\Omega^2+J\theta^2&\equiv 0\mod \theta^1, \theta^3\\
dI_2+5I_2\Omega^1+J\theta^3 &\equiv 0\mod \theta^1, \theta^2
\end{split}
\end{equation}
for some function $J$. Therefore only two among the invariants $I_1$, $I_2$, $I_3$ are functionally independent, namely $I_1$ and $I_2$. 

In the previous structure equations all the non-zero coefficients are differential invariants of the equivalence problem. Moreover we can use them in order to give necessary and sufficient conditions for equivalence under area-preserving transformations to a given equation. Taking into account the condition \eqref{eqha30}, we first have the following proposition
\begin{proposition}
\label{pr}
The invariants of the equation \rm{(\ref{eqha37})} vanish if and only if 
\begin{enumerate}
  \item $f_{y^{\prime}y^{\prime}y^{\prime}y^{\prime}}=0$, ie 
  $$f(x,y,y^{\prime})=A(x,y)y^{\prime3}+3B(x,y)y^{\prime2}+3C(x,y)y^{\prime}+D(x,y)$$
  \item and 
  \begin{equation}
  \label{eqha39}
  \begin{split}
       D_y-C_x&=2(BD-C^2)\\
C_y-B_x&=(AD-BC)\\
B_y-A_x&=2(AC-B^2).
\end{split}
\end{equation}
\end{enumerate}
\end{proposition}
\begin{proof}
From the structure equations (\ref{eqha37}), one sees that  the invariants vanish if and only if
\begin{equation}
\label{eqha40}
\begin{split}
&\dfrac{1}{6}f_{y^\prime y^\prime y^\prime y^\prime}=0\\
&-\dfrac{1}{6}\left(2f_{y^{\prime}y}+\dfrac{1}{3}f_{y^{\prime}}f_{y^{\prime}y^{\prime}}-\mathfrak{D}(f_{y^{\prime}y^{\prime}})\right)=0\\
&f_y+\dfrac{2}{9}f_{y^{\prime}}^2-\dfrac{1}{3}\mathfrak{D}(f_{y^{\prime}})=0\\
&-\dfrac{1}{18}\left(f_{y^{\prime}y{\prime}}^2-2f_{y^{\prime}}f_{y^{\prime}y^{\prime}y^{\prime}}\right)-\dfrac{1}{6}\left(f_{y^{\prime}y^{\prime}y}-\mathfrak{D}\left(f_{y^{\prime}y^{\prime}y^{\prime}}\right)\right)=0.
\end{split}
\end{equation}
The first condition gives us the form
$$f=A(x,y)y^{\prime3}+3B(x,y)y^{\prime2}+3C(x,y)y^{\prime}+D(x,y)$$
for the function $f$. Using this we get three other conditions (by exploiting the just given form of $f$)
 \begin{equation}
 \label{eqha41}
\begin{split}
&(D_y-C_x-2(BD-C^2))+2(C_y-B_x-(AD-BC))y^{\prime}\\
                   &+(B_y-A_x-2(AC-B^2))y^{\prime2}=0\\
&6(C_y-B_x-(AD-BC))+6(B_y-A_x-2(AC-B^2))y^{\prime}=0\\
&A_x-B_y+2(AC-B^2)=0
\end{split}
\end{equation}
with the last relations resulting from the last three relations of equation (\ref{eqha40}) respectively. 
Therefore if the specified conditions hold, all the invariants vanish and conversely as a polynomial in one variable vanishes if and only if its coefficients vanish identically.
\end{proof}
Given a coframe $\Upsilon$ on a manifold $\Sigma$, we remind that its symmetry group or self-equivalences is formed of all the diffeomorphisms of $\Sigma$: $h^\star(\Upsilon)=\Upsilon$ and that the Lie group of special affine transformations is given in standard representation in $M_3(\R)$ by
$$ASL(2,\R)=\biggl\{\left(\begin{array}{cc}1 & 0 \\x & A\end{array}\right), x\in\R^2,A\in SL\rm{(}2,\R\rm{)}\biggl\}.$$
Its Lie-algebra is
$$\mathfrak{asl}(2,\R)\cong \R^2\oplus\mathfrak{sl}(2,\R).$$
We have the following corollary
\begin{corollary}
\label{cor}
A second order ordinary differential equation $y^{\prime\prime}=f(x,y,y^{\prime})$ is equivalent under an area-preserving transformation to $y^{\prime\prime}=0$ if and only if 
\begin{equation}
 \label{eqha42}
\begin{split}
\begin{cases}
f(x,y,y^{\prime})&=A(x,y)y^{\prime3}+3B(x,y)y^{\prime2}+3C(x,y)y^{\prime}+D(x,y)\\
D_y-C_x&=2(BD-C^2)\\
C_y-B_x&=(AD-BC)\\
B_y-A_x&=2(AC-B^2).
\end{cases}
\end{split}
\end{equation}
 Moreover, there is a unique equivalence class of second order ordinary differential equations admitting the maximal Lie group of symmetries $ASL(2,\R)$, namely the equivalence class of $y^{\prime\prime}=0$.
\end{corollary}
\begin{proof}
The conditions given by (\ref{eqha42}) are verified by $y^{\prime\prime}=0$. Since they are equivalent (from the previous proposition) to the vanishing of the invariants of the structure equations (\ref{eqha37}), every equation equivalent to $y^{\prime\prime}=0$ under area-preserving transformations must satisfy them. The second part is proved by observing that the maximal symmetry group of dimension $5$ is achieved if and only if the structure tensor has only constant components, then from equation \eqref{eqha38} we deduce the fact they all vanish hence we end up locally with the Maurer-Cartan structure equations of $ASL(2,\R)$, which are of course invariant under pullback by all automorphisms in the transformation group $ASL(2,\R)$: the left translations $L_a$ with $a$ belonging to $ASL(2,\R)$. 
\end{proof}

\section{Affine normal Cartan connection}
\label{sec4}
We remind that a $G$-principal bundle $P$ is a quadruple $(P,M,G,\pi)$. $P$ is the total space, $M$ the base manifold, $G$ is a Lie group acting on the right on $P$ and $\pi$ the projection is a submersion. We will note it (the principal bundle) $\pi:P\to M$ or $G\to P\to M$.

Let $q\in P$, $x=\pi(q)\in M$, $p\in\pi^{-1}(x)$. Let the group right action be denoted by $R$:
\begin{equation}
\begin{split}
R&:G\times P\to P\\
R(g,p)&=R_g(p):=p.g
\end{split}
\end{equation}

Take $p$ as previously and define the vertical subspace $V_p\subset T_pP$ at $P$ as
$$V_p=\ker \pi_{\star}.$$
A vector field $X$ on $P$ will be called vertical if $X_p\in V_p$ for any $p$. The vertical vector fields form a $G$-invariant involutive distribution as
$$(R_g)_{\star}V_p=V_{p.g}$$
and the Lie bracket of two vertical vector fields is again vertical.

In the absence of any extra structure, there is no natural way to choose a complement to $V_p$ in $T_pP$.  This exactly what a connection provides.
\begin{definition}
A connection on $P$ is a smooth choice of horizontal subspaces $H_p\subset T_pP$ complementary to $V_p$:
$$T_pP=V_p\oplus H_p$$
and such that $(R_g)_{\star}H_p=H_{p.g}$. This also means that a connection is a $G$-invariant distribution $H\subset TP$ complementary to $V$.
\end{definition} 
Let $v\in\mathfrak{g}$ and $\varrho(v)$ the associated vector field on $P$ induced by the $G$-action. One has the following 
 \begin{definition}
 The connection one-form of a connection $H\subset TP$ is the $\mathfrak{g}$-valued one form $\omega$ defined by
  \begin{equation}
  \omega(X)=
\begin{cases}
      v& \text{if $X=\varrho(v)$ }, \\
      0& \text{if $X$ is horizontal}.
\end{cases}
\end{equation}
It obeys to the identity
$$(R_{g})^{\star}\omega=Ad_{g^{-1}}\omega$$
with $Ad$ the adjoint representation of the Lie group $G$.
 \end{definition}
 Now a form on $P$ is said horizontal if it it annihilates the vertical vectors. 
 
 Let $\omega_1$ a $1$-form with values in $\mathfrak{g}$ with and $\omega_2$ another $1$-form with values in $\mathfrak{g}$. We define the $2$-form
 $$\left[\omega_1,\omega_2\right]:(X,Y)\to \left[\omega_1(X),\omega_2(Y)\right]+\left[\omega_2(X),\omega_1(Y)\right]$$
 
 Consider now a connection with form $\omega$. Its curvature is defined by the formula
 \begin{equation}
\Omega=d\omega+\dfrac{1}{2}\left[\omega,\omega\right].
\end{equation}

When $\omega$ is given in matrix form, ie there exists a representation of $\mathfrak{g}$ in some $M_n(\R)$, the curvature takes the form
$$\Omega=d\omega+\omega\wedge\omega$$
with $\omega\wedge\omega$ defined by
$$\omega\wedge\omega(X,Y)=\sum_k\omega_{ik}\wedge\omega_{kj}(X,Y)=(\omega_{ik}(X))(\omega_{kj}(Y))-(\omega_{ik}(Y))(\omega_{kj}(X)).$$
 Next we define with \citep{shar2000} the notions of Cartan geometry and connection

\begin{definition}
\label{defi1}
A Cartan geometry $\xi=(P,\omega,G,H)$ on a manifold $M$ modeled on $(\mathfrak{g},\mathfrak{h})$ with group $H$ \rm{(closed Lie subgroup of $G$)} consists of the following data:
\begin{itemize}
\item a smooth manifold $M$
\item a principal right $H$-bundle $P$ over $M$.
\item a $\mathfrak{g}$-valued $1$-form $\omega$ on $P$ \rm{(}called Cartan connection\rm{)} which satisfies the following conditions:
\begin{enumerate}
\item for each point $p\in P$, the linear map $\omega_p:T_pP\to\mathfrak{g}$ is a linear isomorphism;
\item $(R_h)^{\star}\omega=\rm{Ad}_{h^{-1}}\omega$ for all $h\in H$; 
\item $\omega(\varrho(v))=v$ for all $v\in\mathfrak{h}$ with $\varrho(v)$ the associated vector field.
\end{enumerate}
\end{itemize}
\end{definition}
The curvature of a Cartan connection $\omega$ is defined analogously and verifies the relation
$$\Omega=d\omega+\omega\wedge\omega,$$
when $\omega$ is a matrix of one-forms. If it vanishes the model geometry is called flat.
\begin{remark}
A Cartan connection in $P$ is not a connection in the usual sense since it is not $\mathfrak{h}$-valued. It can however be considered as a connection on an associated bundle to $P$, obtained by enlarging the group structure of $P$, $H$, to the group $G$, ie by considering the principal $G$-bundle
$$P^G=P\times_{H}G.$$
See \rm{\citep{nag1964}}.
\end{remark}
We remind that the $\{e\}$-structure $B_{G_{(2)}^{(1)}}$ is endowed with the coframe $(\theta^1,\theta^2,\theta^3,\Omega^1,\Omega^2)$ defined locally by the equation
\begin{equation}
\label{eqha43}
\begin{split}
\theta^{1}&=u_1\omega^1\\
\theta^{2}&=u_{1}^2\omega^2\\
\theta^{3}&=u_3\omega^1+\dfrac{1}{u_1}\omega^3\\
\Omega^{1}&=\dfrac{du_1}{u_1}+\dfrac{u_3}{u_1}\theta^2+\dfrac{1}{3}u_1f_{y^{\prime}}\theta^3+\dfrac{1}{6u_1}\left(f_{y^\prime y^\prime}-2u_1u_3f_{y\prime}\right)\theta^1\\
\Omega^{2}&=\dfrac{u_3}{u_1}\left(\dfrac{du_1}{u_1}+\dfrac{du_3}{u_3}\right)+\left(\dfrac{u_3}{u_1}\right)^2\theta^2+\dfrac{1}{3}u_3f_{y^\prime}\theta^3\\
&-\dfrac{1}{6u_1}\left(f_{y^{\prime}y^{\prime}}-2u_1u_3f_{y^{\prime}}\right)\theta^3\\
&+\left(-\dfrac{2}{3}\dfrac{u_3^2}{u_1}f_{y^\prime}-\dfrac{1}{6u_1^3}f_{y^{\prime}y^{\prime}y^{\prime}}+\dfrac{1}{2}\dfrac{u_3}{u_1^2}f_{y^{\prime}y^{\prime}}\right)\theta^1.
\end{split}
\end{equation}

Our goal in the remainder of the paper is to show that the coframe above, which defines on $B_{G_{(2)}^{(1)}}$ an $\{e\}$-structure with structure equations \eqref{eqha37}, gives a normal (unique) Cartan connection on $B_{G_{(2)}^{(1)}}$ with values in $\mathfrak{asl}(2,\R)$ and structure group the following closed (in the induced topology of $GL(3,\R))$, $2$-dimensional subgroup of $ASL(2,\R)$ parameterized in the following way
\begin{equation}
\label{eqha44}
H_:=\left\{\left(\begin{array}{ccc}1 & 0 & 0 \\0 & v_1 & -v_3 \\0 & 0 & \dfrac{1}{v_1}\end{array}\right)\right\},
\end{equation}
with $v_1\in\R^\star$, $v_3\in\R$.  

This is how we will tackle the problem.

Let $\pi:B_{G_{(2)}^{(1)}}\to \Ji(\R,\R)$ the canonical projection. Choose a local trivializing open set $U$ of $\Ji(\R,\R)$, which enables us to identify $\pi^{-1}(U)\simeq U\times H$. Assuming we have chosen the open sets $U$ as a locally finite open cover (denoted $U_\alpha)$ of $\Ji(\R,\R)$ with associated partition of unity $f_\alpha$, we get a Cartan connection on $B_{G_{(2)}^{(1)}}$ by means of the formula
\begin{equation}
\label{eqha45}
\omega=\sum_\alpha (f_\alpha\circ\pi)\pi^\star\omega_\alpha
\end{equation}
where $\omega_\alpha$ denotes the Cartan connection on $U_\alpha\times H$.
 
Let $U_\alpha\times H$ with the notation as above. 

Arrange the coframe of \eqref{eqha43} into the following form (the candidate connection)
 \begin{equation}
\label{eqha46}
\begin{split}
\omega_\alpha=\left(\begin{array}{ccc}0 & 0 & 0 \\\theta^3 & \Omega^1 & -\Omega^2 \\\theta^1 & \theta^2 & -\Omega^1\end{array}\right)
\end{split}
\end{equation}
 We first remark that $\omega_\alpha$ is $\mathfrak{asl}(2,\R)$-valued; moreover the structure equations \eqref{eqha37} are equivalent to $d\omega_\alpha+\omega_\alpha\wedge\omega_\alpha$ and the following equation (with the notations of equation \eqref{eqha38}) holds
 \begin{equation}
 \label{eqha47}
 d\omega_\alpha+\omega_\alpha\wedge\omega_\alpha=\left(\begin{array}{ccc}0 & 0 & 0 \\0 & I_1\theta^3\wedge\theta^1 & -I_2\theta^2\wedge\theta^1-I_3\theta^3\wedge\theta^1 \\0 & 0 & -I_1\theta^3\wedge\theta^1\end{array}\right).
 \end{equation}
 Therefore the curvature of $\omega_\alpha$ defines and contains all the area-preserving differential invariants.

We have to verify the $3$ conditions of bullet $3$ of the definition \ref{defi1}.
 
 The first condition is obvious, the mapping $(\omega_\alpha)_p:T_pB_{G_{(2)}^{(1)}}\to \mathfrak{g}$ is injective as $(\theta^1,\theta^2,\theta^3,\Omega^1,\Omega^2)$ is a coframe of $B_{G_{(2)}^{(1)}}$ and $T_pB_{G_{(2)}^{(1)}}$ and $\mathfrak{g}$ have same dimension.
 
  To show the third requirement of bullet $3$ of definition \ref{defi1}, we remark that the fundamental vector field associated to $v\in\mathfrak{h}$ is given in $U_\alpha \times H$ by 
  \begin{equation}
  \label{eqha48}
  \varrho(v)=(0,V)
  \end{equation}
  with $V$ the associated left invariant vector field on $H$. Finally remarking that the left Maurer-Cartan form $\omega_H$ of $H$ is given in coordinates by 
  \begin{equation}
  \label{eqha49}
 \omega_H= \left(\begin{array}{ccc}0 & 0 & 0 \\0 & \frac{du_1}{u_1} & -\frac{u_3}{u_1}\left(\frac{du_1}{u_1}+\frac{du_3}{u_3}\right) \\0 & 0 & -\frac{du_1}{u_1}\end{array}\right)
  \end{equation}
  and applying $\omega$ to $\varrho(v)=(0,V)$ we get 
  \begin{equation}
  \omega_\alpha(\varrho(v))=v=\omega_H(V).
  \end{equation}

  It remains to verify the last condition concerning the equivariant behavior of $\omega_\alpha$ under the right action of $H$.
  
 On $U_\alpha\times H$, $H$ acts in the following way
 
\begin{equation}
\label{eqha50}
\begin{split}
\Phi&:H\times(U\times H)\to U\times H\\
& (T,(p,S))\to (p,ST)=\Phi_T(p,S).
\end{split}
\end{equation}
 If $S=\left(\begin{array}{ccc}1 & 0 & 0 \\0 & u_1 & -u_3 \\0 & 0 & \frac{1}{u_1}\end{array}\right)$ and $T=\left(\begin{array}{ccc}1 & 0 & 0 \\0 & v_1 & -v_3 \\0 & 0 & \frac{1}{v_1}\end{array}\right)$ then we have 
 $$ST=\left(\begin{array}{ccc}1 & 0 & 0 \\0 & u_1v_1 & -u_1v_3-\frac{u_3}{v_1} \\0 & 0 & \frac{1}{u_1v_1}\end{array}\right).$$
 But a straightforward computation gives us 
 \begin{equation}
 \label{eqha51}
 \begin{split}
 Ad_{T^{-1}}\omega_\alpha&=T^{-1}\omega_\alpha T\\
 &=\left(\begin{array}{ccc}0 & 0 & 0 \\\frac{\theta^3}{v_1}+v_3\theta^1 & \Omega^1+v_1v_3\theta^2 & -2\frac{v_3}{v_1}\Omega^1-v_3^2\theta^2-\frac{\Omega^2}{v_1^2}\ \\v_1\theta^1 & v_1^2\theta^2 & -\Omega^1-v_1v_2\theta^2\end{array}\right).
 \end{split}
 \end{equation}
 Also 
 \begin{equation}
 \label{eqha52}
 \Phi_T^\star\theta^1=v_1\theta^1,\, \phi_T^\star\theta^2=v_1^2\theta^2,\,\Phi_T^\star\theta^3=\dfrac{1}{v_1}\theta^3+v_3\theta^1.
 \end{equation}
 Moreover $\Phi_T^\star\omega_H=Ad_{T^{-1}}\omega_H$ because, for $v\in T_SH$, $\omega_H(v)=dL_{ S^{-1}}(v)$; therefore $\omega_H(d\Phi_T v)=dL_{(ST)^{-1}}\circ d\Phi_T (v)=dL_{T^{-1}}\circ d\Phi_T\circ dL_{S^{-1}}(v)=Ad_{T^{-1}}\omega_H(v)$.
 
 Pulling back $\omega_\alpha$ under $\Phi_T$ while using equation \eqref{eqha52} and the just mentioned identity satisfied by $\omega_H$, gives the desired relation namely
 \begin{equation}
 \Phi_T^\star\omega_\alpha=Ad_{T^{-1}}\omega_\alpha.
 \end{equation}

Therefore we have shown to every second order equation \eqref{eqha01} satisfying the condition
$$f_y+\dfrac{2}{9}f_{y^{\prime}}^{2}-\dfrac{1}{3}\mathfrak{D}(f_{y^{\prime}})=0$$
 is associated an $\mathfrak{asl}(2,\R)$ Cartan connection on a $H$ principal bundle $B_{G_{(2)}^{(1)}}$ over the first jet space $\Ji(\R,\R)$.

 We want now to show the Cartan connection $\omega$ of \eqref{eqha46} is the normal (unique) one satisfying certain ''normalization'' conditions.

  In order to do so, we start with a connection matrix $\omega$ on the $H$ principal bundle $B_{G_{(2)}^{(1)}}$ over $\Ji$ with values in the special affine algebra $\mathfrak{asl}(2,\R)$ and work with local coordinates $(x,y,y^{\prime},u_1,u_2)$ compatible with the local trivialization $U_\alpha\times H$. We build the connection on such a trivialization and use a gluing procedure by means of a partition of unity as previously done.
   
 Fix the cross-section $s:U_\alpha\to U_\alpha\times H:p\to(p,e)$ characterized by 
 $$(u_1=1,u_2=0).$$  
 Let $\eta$ the value of the connection $w_\alpha$ on the section $s$. We impose some natural conditions on $\eta$ and then lift it to a connection matrix on $U_\alpha\times H$ via the formula (with $h\in H$ a generic element)

 \begin{equation}
\label{eqha182}
w_\alpha=h^{-1}\eta h+h^{-1}dh.
\end{equation} 

Moreover the curvature is in this case given by the formula
 \begin{equation}
\label{eq194}
dw_\alpha+w_{\alpha}\wedge w_{\alpha}=h^{-1}(d\eta+\eta\wedge\eta)h
\end{equation}

 The given one form $\eta$ can be written in component form as
 
 \begin{equation}
\label{eqha183}
\eta=\left(\begin{array}{ccc}0 & 0 & 0 \\\eta^1 & \eta^1_1 & \eta^1_2 \\\eta^2 & \eta_1^2& \eta_2^2\end{array}\right).
\end{equation}
Now since $\eta$ must be with values in $\mathfrak{asl}(2,\R)$, then 
$$\eta^1_1+\eta^2_2=0.$$
The structure equations of this connection, derived from the structure equations of $\mathfrak{asl}(2,\R)$ are thus given by
\begin{equation}
\label{eqha184}
\begin{split}
d\eta^1&=-\eta^1_1\wedge\eta^1-\eta_2^1\wedge\eta^2+\Theta^1\\
d\eta^2&=-\eta^2_1\wedge\eta^1-\eta^2_2\wedge\eta^2+\Theta^2\\
d\eta_1^1&=-\eta_2^1\wedge\eta_1^2+\Theta_1^1\\
d\eta_1^2&=-2\eta_1^2\wedge\eta_1^1+\Theta_1^2\\
d\eta_2^1&=2\eta_2^1\wedge\eta_1^1+\Theta_2^1.
\end{split}
\end{equation}
We want this connection matrix to define a connection over the open set $U_\alpha$ of $\Ji$; so we choose a basis of one forms $\eta^1$, $\eta^2$, $\Pi$ on $U_\alpha$. This basis has to define the general solutions of equation (\ref{eqha01}) (modulo pullback). 

With the previous specifications, $\Theta^i$ and $\Theta^j_i$ must be horizontal with respect to the fibration $\pi:B_{G_{(2)}^{(1)}}|_{U_\alpha}\to U_\alpha$ ie must be in the algebraic ideal with respect to wedge product, generated by $<\theta^1,\theta^2,\Pi>$. See \citep[prop.2, p.220]{nag1964} or \citep[chap.5, cor.3.10]{shar2000}.

The structure equations (\ref{eqha184}) hence are more explicitly written in the form
\begin{equation}
\label{eqha185}
\begin{split}
d\eta^1&=-\eta^1_1\wedge\eta^1-\eta_2^1\wedge\eta^2+\Theta^1\\
d\eta^2&=-\eta^2_1\wedge\eta^1-\eta^2_2\wedge\eta^2+\Theta^2\\
d\eta_1^1&=-\eta_2^1\wedge\eta_1^2+R_{113}^1\eta^1\wedge\Pi+R_{123}^1\eta^2\wedge\Pi+R_{112}^1\eta^1\wedge\eta^2\\
d\eta_1^2&=-2\eta_1^2\wedge\eta_1^1+R_{113}^2\eta^1\wedge\Pi+R_{123}^2\eta^2\wedge\Pi+R_{112}^2\eta^1\wedge\eta^2\\
d\eta_2^1&=2\eta_2^1\wedge\eta_1^1+R_{213}^1\eta^1\wedge\Pi+R_{223}^1\eta^2\wedge\Pi+R_{212}^1\eta^1\wedge\eta^2.
\end{split}
\end{equation}
We choose $\eta^1=dx$, $\eta^2=dy-y^{\prime}dx$ and $\Pi=dy^{\prime}-fdx+\mu(dy-y^\prime dx)$. This implies $\Pi=\eta^2_1$.

We set now the following "normalization" condition in order to fix the connection:

\begin{equation}
\label{eqha300}
\begin{split}
\Theta^1&=\Theta^2=0\\
R_{113}^1&=R_{123}^1=R_{113}^2=R_{123}^2=0.
\end{split}
\end{equation}
\begin{remark}
Geometrically one can say that the conditions given in the previous equation \eqref{eqha300} correspond to absorption of torsion and normalization in the Cartan equivalence method.
 \end{remark} 
If $\Theta^1$ and $\Theta^2$ vanishes then
$$\eta^1_1\wedge dx+\eta_2^1\wedge(dy-y^\prime dx)=0$$
and
$$-dy^{\prime}\wedge dx=-(dy^\prime-fdx+\mu(dy-y^\prime dx))\wedge dx-\eta_2^2\wedge(dy-y^\prime dx)$$
This gives by applying the Cartan lemma to the first condition while taking into account the second one
\begin{equation}
\label{eqha186}
\begin{split}
\eta_1^1&=-\mu dx+\delta(dy-y^\prime dx)\\
\eta_2^1&=\delta dx+\nu(dy-y^\prime dx)\\
\eta_2^2&=\mu dx-\delta(dy-y^\prime dx)
\end{split}
\end{equation}
for functions $\mu$, $\delta$, $\nu$ on $\Ji$. Using the other "normalization" conditions, we get after computation
\begin{equation}
\label{eqha191}
\begin{split}
R_{113}^1&=\mu_{y^{\prime}}+2\delta\\
R_{123}^1&=\nu-\delta_{y^{\prime}}\\
R_{113}^2&=f_{y^{\prime}}+3\mu\\
R_{123}^2&=-(\mu_{y^{\prime}}+2\delta)
\end{split}
\end{equation}
and this gives
\begin{equation}
\label{eqha187}
\mu=-\dfrac{1}{3}f_{y^\prime},\quad \delta=\dfrac{1}{6}f_{y^\prime y^\prime},\quad\nu=\dfrac{1}{6}f_{y^\prime y^\prime y^\prime}
\end{equation}
In other words we have the following result
\begin{equation}
\label{eqha188}
\begin{split}
\eta^1&=dx\\
  \eta^2&=dy-y^\prime dx\\
\eta^1_1&=\dfrac{1}{3}f_{y^\prime}dx+\dfrac{1}{6}f_{y^\prime y^\prime}(dy-y^\prime dx)\\
\eta^1_2&=\dfrac{1}{6}f_{y^\prime y^\prime}dx+\dfrac{1}{6}f_{y^\prime y^\prime y^\prime}(dy-y^\prime dx)\\
\eta^2_1&=dy^\prime-f dx-\dfrac{1}{3}f_{y^\prime}(dy-y^\prime dx).
\end{split}
\end{equation}
Applying the identity \eqref{eqha182} we get $w_\alpha$ and recover also the local expression of $\omega_\alpha$ of equation \eqref{eqha46}:
$$\omega_\alpha=w_\alpha.$$
Hence the uniqueness of $\omega$. Assuming the above notations, we summarize the whole discussion of the section in the following theorem
\begin{theorem}
\label{th}
To the equivalence of the second order differential equations $y^{\prime\prime}=f(x,y,y^\prime)$ under area-preserving transformations satisfying the relation given by an area-preserving differential invariant:
$$f_y+\dfrac{2}{9}f_{y^{\prime}}^{2}-\dfrac{1}{3}\mathfrak{D}(f_{y^{\prime}})=0,$$
is associated a Cartan geometry on a $5$-dimensional $H$ principal bundle $B_{G_{(2)}^{(1)}}$ over $\Ji(\R,\R)$ with model $(\mathfrak{asl}(2,\R),\mathfrak{h})$. The associated Cartan connection $\omega$, is (unique) normal. Moreover the curvature of $\omega$ contains all the area-preserving differential invariants as given by equation \eqref{eqha47}.

Finally two second order equations $y^{\prime\prime}=f(x,y,y^\prime)$ and $\overline{y}^{\prime\prime}=\overline{f}(\overline{x},\overline{y},\overline{y}^\prime)$ satisfying the relations
$$f_y+\dfrac{2}{9}f_{y^{\prime}}^{2}-\dfrac{1}{3}\mathfrak{D}(f_{y^{\prime}})=0$$
and 
$$\overline{f}_{\overline{y}}+\dfrac{2}{9}\overline{f}_{\overline{y}^{\prime}}^{2}-\dfrac{1}{3}\mathfrak{\overline{D}}(\overline{f}_{\overline{y}^{\prime}})=0$$

respectively, are area-preserving equivalent if and only if there exists a bundle map $F:B_{G_{(2)}^{(1)}}\to \overline{B}_{G_{(2)}^{(1)}}$, ie commuting with the right action of $H$, with obvious notations, such that
$$F^\star\overline{\omega}=\omega.$$ 
\end{theorem}

\bibliographystyle{hplain} 
\bibliography{mybiblio3}

\end{document}